\documentclass{amsart}

\usepackage{epsfig}
\usepackage{amsmath}
\usepackage{amssymb}
\usepackage{amscd}
\usepackage{graphicx}
\usepackage{color}
\usepackage{tikz}
\usetikzlibrary{arrows,shapes,positioning,shadows,trees}
\tikzset{
  basic/.style  = {draw, text width=2cm, drop shadow, font=\sffamily, rectangle},
  root/.style   = {basic, rounded corners=2pt, thin, align=center,
                   fill=white!30, text width=12em},
  level 2/.style = {basic, rounded corners=6pt, thin,align=center, fill=white!60,
                   text width=12em},
  level 3/.style = {basic, rounded corners=6pt, thin,align=center, fill=white,
                   text width=9em}
}

\newtheorem{thm}{Theorem}[section]
\newtheorem{lem}[thm]{Lemma}
\newtheorem{prop}[thm]{Proposition}

\newtheorem{cor}[thm]{Corollary}

\newtheorem{rem}[thm]{Remark}

\newtheorem{exam}[thm]{Example}
\newtheorem*{thm AY}{Auslander-Yorke Dichotomy Theorem}
\newtheorem*{quest}{Question}

\newcommand{\diam}{{\mathrm{diam}}}
\newcommand{\dist}{{\mathrm{dist}}}
\newcommand{\Tran}{{\mathrm{Tran}}}
\newcommand{\orb}{{\mathrm{orb}}}
\newcommand{\inter}{{\mathrm{int}}}

\newcommand{\Eq}{{\mathrm{Eq}}}

\numberwithin{equation}{section}

\makeatletter
\@namedef{subjclassname@2010}{%
  \textup{2010} Mathematics Subject Classification}
\makeatother

\begin{document}

\title{analogues of Auslander--Yorke theorems for multi--sensitivity}

\subjclass[2010]{}

\author{Wen Huang, Sergi{\u\i} Kolyada and Guohua Zhang}

\address{Department of Mathematics, Sichuan University,
Chengdu, Sichuan 610064, China}

\address{School of Mathematical Sciences, University of Science and Technology of
China, Hefei, Anhui 230026, China}

\email{wenh@mail.ustc.edu.cn}

\address{Institute of Mathematics, NASU, Tereshchenkivs'ka 3, 01601 Kyiv, Ukraine}

\email{skolyada@imath.kiev.ua}

\address{School of Mathematical Sciences and LMNS, Fudan University and Shanghai Center for Mathematical Sciences, Shanghai 200433, China}

\email{chiaths.zhang@gmail.com}

\begin{abstract}
We study multi-sensitivity and thick sensitivity for continuous surjective selfmaps on
compact metric spaces.
Our main result states that a minimal system is either multi-sensitive or an almost one-to-one extension of its maximal
equicontinuous factor. This is an analog of the Auslander-Yorke dichotomy theorem: a minimal system is either sensitive or equicontinuous.
Furthermore, we introduce the concept of a syndetically equicontinuous point, and prove that a transitive system is either thickly sensitive
or contains syndetically equicontinuous points, which is a refinement of another well known result of Akin, Auslander and Berg.
\end{abstract}



\maketitle

\markboth{}{}


\section{Introduction}

Throughout this paper $(X,T)$ denotes a  \textit{topological dynamical system}, where
$X$ is a compact metric space with metric $\varrho$ and $T:X \to X$ is a continuous surjection.

\medskip

The notion of sensitivity (sensitive dependence on initial conditions) was first used by Ruelle \cite{Ru}.
According to  the works by Guckenheimer \cite{Gu}, Auslander and Yorke \cite{AuslanderYorke} a dynamical
system $(X,T)$ is called
{\it sensitive} if there exists $\delta> 0$ such that for every $x\in X$ and every
neighborhood $U_x$ of $x$, there exist $y\in U_x$ and $n\in \mathbb{N}$ with $ \varrho(T^n(x),T^n(y))> \delta$, where $\mathbb{N}$ is the set of all natural numbers (positive integers).
According to \cite{AK}, it is easy to see that $(X,T)$ is sensitive if and only if $S_T (U, \delta)$
is infinite for some $\delta> 0$ and every opene\footnote{Because we so often have to refer to
open, nonempty subsets, we will call such subsets \emph{opene}.} set $U\subset X$, where
$$S_T (U, \delta)= \{n\in \mathbb{N}: \text { there are } x_1, x_2\in U\ \text{such that}\ \varrho (T^n x_1, T^n x_2)>
\delta\}.$$
We define $J_T (U, \delta)\subset \mathbb{N}$ to be the complement of $S_T (U, \delta)$.

The Lyapunov stability or, in other words, equicontinuity is the opposite to the notion of sensitivity.
Recall that a point $x \in X$ is called \emph{Lyapunov stable}  if for every $\varepsilon > 0$ there exists a $\delta > 0$ such that $\varrho(x,x') < \delta$
implies $\varrho (T^n x, T^n x')< \varepsilon$ for any $n\in \mathbb{N}$, equivalently, for every $\varepsilon> 0$ there exists a neighborhood $U$ of $x$ such that $J_T (U, \varepsilon)= \mathbb{N}$.
This condition says exactly that the sequence of iterates $\{T^{n} : n \geq 0 \}$ is
equicontinuous at $x$.
Hence, such a point is also called an \emph{equicontinuity point} of $(X, T)$.
Denote by $\text{Eq} (X, T)$ the set of all equicontinuity points of $(X, T)$.
The system $(X, T)$ is called \emph{equicontinuous} if $\text{Eq} (X, T)= X$.

The well-known Auslander-Yorke dichotomy theorem states that a
minimal dynamical system is either sensitive or equicontinuous
 \cite{AuslanderYorke} (see also \cite{GlasnerWeiss1993}), which was further refined in \cite{Akin1997, AAB1993}:
 a transitive system is either sensitive or almost equicontinuous (in the sense of containing some equicontinuity points). We recommend \cite{LiYe} for a survey on the recent development of chaos theory, including sensitivity and equicontinuity, in topological dynamics.

\medskip

Recall that a subset $\mathcal{S}\subset \mathbb{N}$ is called \emph{thick} if for each $k\in \mathbb{N}$ there
 exists $n_k\in \mathbb{N}$ such that $\{n_k, n_k+ 1, \dots, n_k+ k\}\subset \mathcal{S}$,
 and is
 \emph{syndetic} if there exists $m\in \mathbb{N}$ such that
 $\mathcal{S}\cap \{n, n+ 1, \dots,
  n+ m\}\not = \varnothing$ for each $n\in \mathbb{N}$. A thick set has a nonempty intersection with every syndetic set. The following definitions of stronger
 forms of sensitivity were introduced in \cite{liuheng, Subrahmonian2007}.
A dynamical system $(X,T)$ is said to be
\begin{itemize}
\item[(1)] \emph{thickly sensitive} if there exists $\delta> 0$ such that $S_T (U, \delta)$
is thick for any opene $U\subset X$;

\item[(2)] \emph{multi-sensitive}
if there exists $\delta> 0$ such that $\bigcap\limits_{i= 1}^k S_T (U_i, \delta)\neq \varnothing$  for any finite collection $U_1, \dots, U_k$ of opene subsets of $X$.
\end{itemize}

In the paper we show that an analog of the Auslander-Yorke dichotomy theorem can also be found for this stronger forms of sensitivity.
Precisely, by using Veech's characterization of
 equicontinuous structure relation of a system \cite[Theorem 1.1]{Veech1968}, we prove that
a minimal system is either thickly sensitive or an almost one-to-one extension of
its maximal equicontinous factor (Theorem \ref{1407120539}),
and thick sensitivity is equivalent to multi-sensitivity
for transitive systems (Proposition \ref{1407052203}).
In particular,
 an invertible minimal system is either multi-sensitive or almost automorphic (Corollary \ref{1407161841}).

 Recall that the concept of almost automorphy, as a generalization of almost periodicity, was first introduced
 by Bochner in 1955 (in the context of differential geometry \cite{Bochner}) and
 studied by many authors starting from \cite{Bo}, \cite{Veech1965}, \cite{Veech1977}.


We extend the notion of equicontinuity by demanding that the set $J_T (U, \varepsilon)$ is large. Precisely, we introduce the concept of syndetically equicontinuous points, which means that for every $\varepsilon> 0$ there exists a neighborhood $U$ of $x$ such that $J_T (U, \varepsilon)$ is a syndetic set.
It turns out that this new notion of local equicontinuity is very useful. That is, the refined Auslander-Yorke
dichotomy theorem \cite{Akin1997, AAB1993} also holds in our setting (Theorem \ref{1407181746}):
a transitive system is either thickly sensitive or contains syndetically equicontinuous points.

 \medskip
\noindent {\bf Acknowledgements.}  The authors acknowledge the hospitality
of the Max-Planck-Institute f\"ur Mathematik (MPIM) in Bonn, where a
substantial part of this paper was written during the Activity ``Dynamics and Numbers'', June -- July 2014.
We thank MPIM for providing an ideal setting. The first and second authors also
acknowledge the hospitality of the School of Mathematical Sciences of
the Fudan University, Shanghai.  We also thank Xiangdong Ye for sharing his joint
work \cite{LiYe, YeYu}, Joseph Auslander for his short proof of Lemma \ref{added}, T. K. S. Moothathu
for his valuable comments, and the referees for important comments that have resulted in substantial improvements to this paper.

The first author was supported by NNSF of China (11225105, 11431012), the third author
was supported by NNSF of China (11271078).

\section{Preliminaries}

In this section we recall standard concepts and results used in later discussions.

\subsection{Topological dynamics}
Recall that $(X, T)$ is (\textit{topologically}) \emph{transitive} if $N_T (U_1, U_2)= \{n\in \mathbb{N}: U_1\cap T^{-n}U_2\neq \varnothing \}$
 is nonempty  for any opene $U_1, U_2\subset X$. A point $x\in X$ is called \emph{transitive} if its \emph{orbit} $\orb_T (x)=
   \{T^n x: n= 0, 1, 2, ... \}$ is dense in $X$. Denote by $\Tran (X, T)$ the set of all transitive
    points of $(X, T)$. Since $T$ is surjective,  $(X, T)$ is transitive if and only if $\Tran (X, T)\neq \varnothing $.

    The system $(X, T)$ is called \emph{minimal} if $\Tran(X,T) = X$.
    In general, a subset $A$ of $X$ is \emph{invariant} if $TA = A$.
    If $A$ is a closed, nonempty, invariant subset then $(A,T|_A)$ is called the associated \emph{subsystem}.
    A \emph{minimal subset} of $X$ is a nonempty, closed, invariant subset such that the associated subsystem is minimal.
    Clearly, $(X,T)$ is minimal if and only if it admits no a proper, nonempty, closed, invariant subset.
     A point $x \in X$ is called \emph{minimal} if it lies in some minimal subset.
     Zorn's Lemma implies that every closed, nonempty invariant set contains a minimal set.
Observe that by the classic result of Gottschalk,
 $x\in X$ is minimal if and only if $N_T (x, U)= \{n\in \mathbb{N}: T^n x\in U\}$ is syndetic for any neighborhood $U$ of $x$.

A transitive system $(X, T)$ is called an \emph{E-system} \cite{GlasnerWeiss1993} if it admits an invariant probability Borel measure $\mu$ with
 full support, that is $T \mu= \mu$ and $\mu (U)> 0$ for all opene $U\subset X$.
   Note that a nonminimal E-system is sensitive
\cite[Theorem 1.3]{GlasnerWeiss1993}.

 \subsection{Extensions and factor maps}
  Let $(X, T)$ and $(Y, S)$ be topological dynamical systems. By a \emph{factor map} $\pi: (X, T)\rightarrow (Y, S)$ we mean that $\pi: X\rightarrow Y$ is a continuous
 surjection with $\pi\circ T= S\circ \pi$. In this case, we call $(X, T)$  an \emph{extension} of
  $(Y, S)$ and $(Y, S)$ a \emph{factor} of $(X, T)$, we also call $\pi: (X, T)\rightarrow (Y, S)$
    an \emph{extension}.

Each dynamical system admits a maximal equicontinuous factor. In fact, this factor is related to the regionally
proximal relation of the system.
The \emph{regionally proximal relation} $Q_+ (X, T)$ of $(X, T)$ is defined as: $(x, y)\in Q_+ (X, T)$ if and only if
for any $\varepsilon> 0$ there exist $x', y'\in X$ and $n\in \mathbb{N}$ with $\max \{\varrho (x, x'), \varrho (y, y'),
 \varrho (T^n x', T^n y')\}< \varepsilon$. Observe that $Q_+ (X, T)\subset X\times X$ is closed and positively invariant (in the sense
  that if $(x, y)\in Q_+ (X, T)$ then $(T x, T y)\in Q_+ (X, T)$), and the quotient by the smallest closed, positively invariant equivalence relation containing it is the maximal equicontinuous factor
  $(X_{\text{eq}}, T_{\text{eq}})$ of $(X, T)$. If $(X, T)$ is minimal, then $Q_+ (X, T)$ is in fact an equivalence relation by \cite{Auslander1988, BHM2000, EG1960, Veech1968}
   and \cite[Proposition A.4]{HuangYe2000}.
Denote by $\pi_{\text{eq}}: (X, T)\rightarrow (X_{\text{eq}}, T_{\text{eq}})$ the corresponding factor map. Remark that
 $(X_{\text{eq}}, T_{\text{eq}})$ is invertible, when $(X, T)$ is transitive, because each transitive equicontinuous
 system is uniformly rigid \cite[Lemma 1.2]{GlasnerWeiss1993} and this implies invertibility.

    Let $X$ be a compact metric space and
let $\phi: X\rightarrow Y$ be a continuous surjective map.
Denote by $Y_0\subset Y$ the set of all points $y\in Y$ whose fibers are singletons. The
set $Y_0$ is a $G_\delta$ subset of $Y$,
 because
$$Y_0= \{y\in Y: \phi^{- 1} (y)\ \text{is a singleton}\}= \bigcap_{n\in \mathbb{N}} \left\{y\in Y: \diam (\phi^{- 1} (y))<
\frac{1}{n}\right\}$$
and the map $y\mapsto \diam (\phi^{- 1} (y))$ is upper semi-continuous.
 Recall that
the function $f: Y\rightarrow \mathbb{R}_+$ is
\emph{upper semi-continuous} if $\limsup\limits_{y\rightarrow y_0} f (y)\le f (y_0)$ for each $y_0\in Y$.

If $Y_0\subset Y$ is a dense subset, then we call $\phi$ \emph{almost one-to-one}.
If $\pi: (X, T)\rightarrow (Y, S)$ is an almost one-to-one factor map between topological dynamical systems,
then we also call $(X, T)$ an \emph{almost one-to-one extension of $(Y, S)$}. Recall that if a
dynamical system $(X,T)$ is minimal, where $X$ is a compact metric space, then the map $T:X\to X$ is almost one-to-one
\cite[Theorem 2.7]{KST2001}.

\begin{rem} \label{supplement}
Here we take the definition of \emph{almost one-to-one} from \cite{Downarowicz2005, KST2001}, as we shall use this one in
Proposition \ref{201603271142} and in the construction of Example \ref{1407142346}. Note that the denseness of $\pi^{- 1} (Y_0)$ in $X$,
used heavily in \cite{AkinGlasner2001}, is a sufficient but not necessary condition; and these two conditions are equivalent,
when  $\pi$ is a factor map between minimal systems (see Proposition \ref{201603271632}).
\end{rem}

A pair of points $x, y \in X$  is called \emph{proximal} if $\liminf\limits_{n\rightarrow \infty} \varrho (T^n x, T^n y)= 0$.
In this case each of points from the pair is said to be \emph{proximal} to another.
Let $\pi: (X, T)\rightarrow (Y, S)$ be a factor map between dynamical systems. We call $\pi$ \emph{proximal}
if any pair of points $x_1, x_2\in X$ is proximal whenever $\pi (x_1)= \pi (x_2)$.
Note that any almost one-to-one extension between minimal systems is a proximal extension.

     Recall that the \emph{natural extension} $(\widehat{X}, \widehat{T})$ of $(X, T)$ is defined as
$$\widehat{X}= \{(x_1, x_2, \dots): T (x_{i+ 1})= x_i\ \text{and}\ x_i\in X\ \text{for each}\ i\in \mathbb{N}\},$$
$$\widehat{T}: (x_1, x_2, \dots)\mapsto (T x_1, x_1, \dots),$$
with a compatible metric $d$ given by
$$d ((x_1, x_2, \dots), (x_1^*, x_2^*, \dots))= \sum_{n\in \mathbb{N}} \frac{\varrho (x_n, x_n^*)}{2^n M}\ \text{with}\ M= \diam (X)+ 1.$$
Then $(\widehat{X}, \widehat{T})$ is an invertible extension of $(X, T)$ with a factor map $\widehat{\pi}: (\widehat{X},
 \widehat{T})\mapsto (X, T)$, given by $(x_1, x_2, \dots)\mapsto x_1$.
 Observe that
 $$\left(\prod_{i=1}^n U_i\times \prod_{n+ 1}^\infty X\right)\cap \widehat{X}= \left(\prod_1^{n- 1} X\times \bigcap_{i= 1}^n T^{- (n- i)} U_i\times \prod_{n+ 1}^\infty X\right)\cap \widehat{X}$$
  for any opene $U_1, \dots, U_n\subset X$,
 and all such subsets form a basis for the topology of $\widehat{X}$. Applying this fact, it is not hard to check from the definitions that $(X, T)$ is
  minimal (sensitive, thickly sensitive, multi-sensitive, respectively)
  if and only if $(\widehat{X}, \widehat{T})$ is minimal (sensitive, thickly sensitive, multi-sensitive, respectively).

\subsection{Other concepts}
    A continuous map $\phi: X\rightarrow Y$ is called
\emph{almost open} if $\phi (U)$ has a nonempty interior in $Y$ for any opene $U\subset X$.
Recall that if a system $(X,T)$ is minimal then the map $T:X\to X$ is almost open \cite{KST2001}. It is easy to see that all of sensitivity, thick sensitivity and
  multi-sensitivity can be lifted from a factor to an extension by an almost open factor map by the method used in the
   proof of \cite[Lemma 1.6]{GlasnerWeiss1993}. Note that any factor map from a system containing a dense set of
   minimal points
    to a minimal system is almost open, as each factor map between minimal systems is also almost open \cite[Theorem 1.15]{Auslander1988}.

\begin{lem} \label{201603271750}
Let $\pi: (X, T)\rightarrow (Y, S)$ be a factor map between minimal systems. If $Y_1\subset Y$ is a dense subset, then $\pi^{- 1} (Y_1)\subset X$ is also a dense subset.
\end{lem}
\begin{proof}
If the conclusion does not hold, then $U\subset X$ is an opene subset, where $U$ is the complement of the closure of $\pi^{- 1} (Y_1)$ in $X$. Thus $\pi (U)$ has a nonempty interior in $Y$ by \cite[Theorem 1.15]{Auslander1988}, and hence $\pi (U)\cap Y_1\neq \emptyset$ by the denseness of $Y_1$ in $X$, which implies $U\cap \pi^{- 1} (Y_1)\neq \emptyset$, a contradiction with the construction of $U$.
\end{proof}

Now we can prove the following

\begin{prop} \label{201603271632}
Let $\pi: (X, T)\rightarrow (Y, S)$ be a factor map between minimal systems. Denote by $Y_0$ the set of all points $y\in Y$ whose fibers $\pi^{- 1} (y)$ are singletons. Then $Y_0$ is a dense subset of $Y$ if and only if $\pi^{- 1} (Y_0)$ is a dense subset of $X$.
\end{prop}
\begin{proof}
It suffices to show the denseness of $\pi^{- 1} (Y_0)$ in $X$ when $Y_0$ is dense in $Y$. Denote by $Y_1$ the set of all points $y\in Y$ such that $S^{- 1} (y)$ is a singleton.
Then $Y_1\subset Y$ is a dense $G_\delta$ subset by \cite[Theorem 2.7]{KST2001}; and hence so is $S^{- n} (Y_1)\subset Y$ for each $n\in \mathbb{Z}_+$, whose denseness in $Y$ follows from Lemma \ref{201603271750}.
Now set
\begin{equation*} \label{201603271650}
Y_*= Y_0\cap Y_\infty\ \text{with}\ Y_\infty= \bigcap_{n\ge 0} S^{- n} (Y_1).
\end{equation*}
We have that $Y_*\subset Y$ is a dense $G_\delta$ subset by the assumption of $Y_0$.

Let $y\in Y_*$, and assume $\pi^{- 1} (y)= \{x^*\}$ by the definition of $Y_0$. It is clear that $S y\in Y_\infty\subset Y_1$, in particular, $S^{- 1} (S y)= \{y\}$. Now if $x\in X$ satisfies $\pi (x)= S y$, and take $x_0\in T^{- 1} (x)$, then $S (\pi x_0)= \pi (T x_0)= S y$, which implies $\pi (x_0)= y$ and hence $x_0= x^*$, thus $x= T x_0= T x^*$. That is, $\pi^{- 1} (S y)= \{T x^*\}$, which implies $S y\in Y_0$ and hence $S y\in Y_*$. This show that $Y_*$ is a positively invariant subset of $Y$, and hence $\pi^{- 1} (Y_*)$ is a positively invariant subset of $X$. Finally applying the minimality of the system $(X, T)$ we obtain the denseness of $\pi^{- 1} (Y_*)$ in $X$, and then the denseness of $\pi^{- 1} (Y_0)$ in $X$. This finishes the proof.
\end{proof}

We  also use the following concepts by Furstenberg \cite{Furstenberg1981}.
Let $\mathcal{S}\subset \mathbb{N}$. The set $\mathcal{S}$ is called a \emph{central set} if there exists a topological dynamical
system $(X, T)$ with $x\in X$ and open $U\subset X$ containing a minimal point $y$ of $(X, T)$ such that the pair
$(x, y)$
is proximal and $N_T (x, U)\subset \mathcal{S}$. The set $\mathcal{S}$ is called a \emph{difference set} or shortly \emph{$\Delta$-set} if there exists
$\{s_1< s_2< \dots\}\subset \mathbb{N}$
with $\mathcal{S}= \{s_i- s_j: i> j\}$. The set $\mathcal{S}$ is a \emph{$\Delta^*$-set} if it
has a nonempty intersection with any $\Delta$-set.
Note that each central set contains
 a $\Delta$-set \cite[Proposition 8.10 and Lemma 9.1]{Furstenberg1981}; and if $(X, T)$ is a minimal system, then $N_T (U, U)$ is
 a $\Delta^*$-set for any opene $U\subset X$ by \cite[Page 177]{Furstenberg1981}.

 \section{Dichotomy of multi-sensitivity for transitive systems}

 The Auslander-Yorke dichotomy theorem states that a minimal system is either sensitive or equicontinuous
  (see \cite{Akin1997, AAB1993, AuslanderYorke, GlasnerWeiss1993}). The goal of this section is to provide
  an analog of the Auslander-Yorke theorem for multi-sensitivity (see Theorem \ref{1407120539}, Proposition \ref{1407052203} and Theorem \ref{1407181746}),
  which is the main result of this paper.

  \medskip

Our dichotomy is stated firstly for minimal thickly sensitive systems as follows.
Remark that recently Ye and Yu introduced and discussed block sensitivity and strong sensitivity for several families,
and obtained results similar to Theorem \ref{1407120539} for these sensitivities \cite{YeYu}.
We defer the long proof of it to Section 4.

\begin{thm} \label{1407120539}
Let $(X, T)$ be a minimal system. Then $(X, T)$ is not thickly sensitive if and only if $(X, T)$ is an almost one-to-one extension of $(X_{\text{eq}}, T_{\text{eq}})$.
\end{thm}

As shown by the following result, for transitive systems thick sensitivity is equivalent to multi-sensitivity.
Observe that Moothathu pointed out firstly in \cite{Subrahmonian2007} that multi-sensitivity implies thick sensitivity.

\begin{prop} \label{1407052203}
If $(X, T)$ is multi-sensitive, then $(X, T)$ is thickly sensitive. Moreover, if $(X, T)$ is transitive, then the converse also holds.
\end{prop}
\begin{proof}
First assume that $(X, T)$ is multi-sensitive with a sensitivity constant $\delta> 0$, and take any opene $U\subset X$. Let $k\in \mathbb{N}$. For each $i= 0, 1, \cdots, k$, we choose opene $U_i\subset T^{- i} U$ such that $\max\limits_{0\le j\le k} \text{diam} (T^j U_i)< \delta$. By the assumption of $\delta$ we may select $n_k\in \bigcap_{i= 0}^k S_T (U_i, \delta)$. Moreover, from the construction of $U_0, U_1, \dots, U_k$ one has that $n_k\in \bigcap_{i= 0}^k S_T (T^{- i} U, \delta)$ and $n_k> k$. Obviously, $\{n_k- k, \dots,
 n_k- 1, n_k\}\subset S_T (U, \delta)$, which implies that $(X, T)$ is thickly sensitive with a sensitivity constant $\delta$.

Now we assume that a transitive system $(X, T)$ is thickly sensitive with a sensitivity constant $\delta> 0$. Let $k\in \mathbb{N}$ and  $U_1, \dots, U_k$ be opene sets in $X$. Take a transitive point
$x\in \Tran (X, T)$. Then  there exists $n_i\in \mathbb{N}$ such that  $T^{n_i} x\in U_i$, where $i= 1, \dots, k$. So, we may get
 an opene $U\subset X$ such that $T^{n_i} U\subset U_i$ for every $i= 1, \dots, k$. By assumption there exists $s\in
  \mathbb{N}$ with $\{s, s+ 1, \dots, s+ n_1+ \dots + n_k\}\subset S_T (U, \delta)$, and then one has
  $s\in \bigcap\limits_{i= 1}^k S_T (U_i, \delta)$. This shows that $(X, T)$ is multi-sensitive with a sensitivity constant $\delta$.
\end{proof}

Let $(X, T)$ be an invertible system. Recall that $x\in X$ is an \emph{almost automorphic point} of $(X, T)$ if $T^{n_k} x\rightarrow x'$
implies $T^{- n_k} x'\rightarrow x$ for any $\{n_k: k\in \mathbb{N}\}\subset \mathbb{Z}$, where $\mathbb{Z}$ is the set of all integers. The system $(X, T)$ is said to be
\emph{almost automorphic} if $X= \overline{\orb_T (x)}$ for an almost automorphic point $x\in X$.
The structure of minimal almost automorphic systems was
 characterized in \cite{Veech1965}: a minimal invertible system is almost automorphic if and only if it is an almost one-to-one extension
  of its maximal equicontinuous factor $(X_{\text{eq}}, T_{\text{eq}})$.

   Thus, directly from Theorem \ref{1407120539}, we have the following

  \begin{cor} \label{1407161841}
Let $(X, T)$ be an invertible minimal system. Then $(X, T)$ is not multi-sensitive if and only if it is almost automorphic.
  \end{cor}

We are going to link thick sensitivity with local equicontinuity of points by introducing the concept of syndetically equicontinuous points of a system.

We say that $x\in X$ is \emph{syndetically equicontinuous} if for any $\varepsilon> 0$ there exists a neighborhood $U$
 of $x$ such that $J_T (U, \varepsilon)$ is a syndetic set. Denote by $\text{Eq}_{\text{syn}} (X, T)$ the set of all syndetically equicontinuous points of $(X, T)$. Then  $\text{Eq}_{\text{syn}} (X, T)\supset \text{Eq} (X, T)$.
  Since a thick set has a nonempty intersection with every syndetic set, one has readily that
 if $(X, T)$ is
 thickly sensitive then $\text{Eq}_{\text{syn}} (X, T)= \varnothing$, equivalently, if $\text{Eq}_{\text{syn}} (X, T)\neq \varnothing $ then $(X, T)$ is
 not thickly sensitive.



Recall the \emph{Auslander-Yorke Dichotomy Theorem} from \cite{AuslanderYorke} as follows, supplemented by some results from \cite{GlasnerWeiss1993}
and \cite{AAB1993}.

\begin{thm AY}  Let $(X,T)$ be a transitive system. Then exactly one of the following two cases holds.

{\bf $\Eq(X,T) \not= \varnothing$:} \emph{Assume that there exists an equicontinuity point for the system}.
The equicontinuity points are exactly the transitive points, i.e., $\Eq(X,T) = \Tran(X,T)$, and the system is almost equicontinuous.
The map $T$ is a homeomorphism and the inverse system $(X,T^{-1})$ is also almost equicontinuous.  Furthermore,
the system is \emph{uniformly rigid} meaning that some subsequence of $\{ T^{n} : n= 0, 1, \dots\}$ converges uniformly to the identity map on $X$.

{\bf $\Eq(X,T) = \varnothing$:} \emph{Assume that the system has no equicontinuity points.}
The system is sensitive.
\end{thm AY}

Similarly, we have the following dichotomy. We call the system $(X, T)$ \emph{syndetically equicontinuous} if $\text{Eq}_{\text{syn}} (X, T)= X$.

\begin{thm} \label{1407181746}
Let $(X, T)$ be a transitive system.
Then either $(X, T)$ is thickly sensitive and so $\text{Eq}_{\text{syn}} (X, T)= \emptyset$, or $\Tran (X, T)\subset \text{Eq}_{\text{syn}} (X, T)$.
In particular, if $(X, T)$ is minimal then it is either thickly sensitive or syndetically equicontinuous.
\end{thm}
\begin{proof}
It suffices to show that, if $(X, T)$ is not thickly sensitive then
 $\Tran (X, T)\subset \text{Eq}_{\text{syn}} (X, T)$.
Let $\delta> 0$. By the assumption there exists opene $U'\subset X$ such that $S_T (U', \delta)$ is not
 thick. Equivalently, there exists syndetic $\mathcal{N}\subset \mathbb{N}$ such that $\varrho (T^n x_1, T^n x_2)\le \delta$
 whenever $x_1, x_2\in U'$ and $n\in \mathcal{N}$. Now for any $x\in \Tran (X, T)$ there exists $m\in \mathbb{N}$ with
  $T^m x\in U'$ and hence there exists open $U\subset X$ containing $x$ with $T^m U\subset U'$. In particular, $\varrho (T^{m+ n} x,
  T^{m+ n} x')\le \delta$ whenever $x'\in U$ and $n\in \mathcal{N}$. That implies $x\in \text{Eq}_{\text{syn}} (X, T)$, as
   $m+ \mathcal{N}$ is a syndetic set.
\end{proof}

Note that the dichotomy Theorem \ref{1407181746} does not work when the system $(X, T)$ is not transitive. The following example is due to an anonymous referee of the paper.

\begin{exam} \label{201603261744}
There exists a system $(X, T)$, which is not thickly sensitive, such that $\text{Eq}_{\text{syn}} (X, T)= \emptyset$.
\end{exam}
\begin{proof}[Construction]
Let $(Y_1, S_1)$ be the full shift homeomorphism on two symbols, that is, $Y_1= \{0, 1\}^\mathbb{Z}$ and $S_1: (y_i: i\in \mathbb{Z})\mapsto (y_{i+ 1}: i\in \mathbb{Z})$. We take a fixed point $e_1\in Y_1$. Let $(Y_2, S_2)$ be the identity map on the one-point compactification of the discrete space $\mathbb{N}$ with $e_2\in Y_2$ the point at infinity. Now we set $(X, T)$ to be the factor of the product system $(Y_1\times Y_2, S_1\times S_2)$ by collapsing $(Y_1\times \{e_2\})\cup (\{e_1\}\times Y_2)$ into a fixed point $e\in X$. Then $(X, T)$ is the required system:

 By the construction of $X$, for each $\delta> 0$, there exists $p\in \mathbb{N}$ such that $Y_1\times \{p\}\subset X$ is an open invariant subset with $\diam (Y_1\times \{p\})< \delta$, in particular, the set $S_T (Y_1\times \{p\}, \delta)= \emptyset$ is not thick. This implies that $(X,T)$ is not thickly sensitive.

 Let $x\in X$. By the construction of $X$, there exists $q\in \mathbb{N}$ such that any neighborhood of $x$ has a nonempty intersection with $Y_1\times \{q\}$, and set $\delta_x= \frac{1}{2} \diam (Y_1\times \{q\})> 0$. It is easy to check that, for each neighborhood $U$ of $x$ there exists $N_U\in \mathbb{N}$ such that $\diam (T^n U)> \delta_x$ for all $n> N_U$, in particular, the set $J_T (U, \delta_x)\subset \{1, \cdots, N_U\}$ is not syndetic. This shows $x\notin \text{Eq}_{\text{syn}} (X, T)$, and then $\text{Eq}_{\text{syn}} (X, T)= \emptyset$.
\end{proof}

The following examples show  that for a transitive not thickly sensitive system $(X, T)$
the nonempty set $\text{Eq}_{\text{syn}} (X, T)$ may be very complicated. Note that by \cite[Theorem 1.3]{GlasnerWeiss1993} each non sensitive E-system is necessarily a minimal equicontinuous system.


\begin{exam} \label{1407142346}
There exists a sensitive, not thickly sensitive system $(X, T)$ such that $\Tran (X, T)\subsetneq \text{Eq}_{\text{syn}} (X, T)= X$. In fact, the constructed system $(X, T)$ is a nonminimal E-system.
\end{exam}
\begin{proof}[Construction]
We take a Toeplitz flow $(Y, S)$ which is a minimal invertible system with positive topological entropy. See
\cite{Downarowicz2005}
 for the construction of such a system.
Let $\pi_{\text{eq}}: (Y, S)\rightarrow (Y_{\text{eq}}, S_{\text{eq}})$
be the factor map of $(Y, S)$ over its maximal equicontinuous factor. Then $\pi_{\text{eq}}$ is an almost one-to-one extension between minimal systems. And hence $(Y, S)$ is not thickly sensitive by Theorem \ref{1407120539}, and $\pi_{\text{eq}}$ is a proximal extension.

By the classical variational principle (see for example \cite[Theorem 8.6]{Walter}) we choose an ergodic invariant Borel probability measure $\nu$ of $(Y, S)$ with positive measure-theoretic $\nu$-entropy $h_\nu (Y, S)$.
 Let $\nu= \int_Z \nu_z d \eta (z)$ be the disintegration of $\nu$ over the Pinsker factor $(Z, \mathcal{D}, \eta, R)$ of
 $(Y, \mathcal{B}_\nu, \nu, S)$, where $(Y, \mathcal{B}_\nu, \nu)$ is the completion of $(Y, \mathcal{B}_Y, \nu)$ and
 $\mathcal{B}_Y$ denotes the Borel $\sigma$-algebra of $Y$  (for the construction of such a disintegration see for example \cite[Chapter 5, \S 4]{Furstenberg1981}). Set $\lambda= \int_Z \nu_z\times \nu_z d \eta (z)$, which is in fact
 an ergodic invariant Borel probability measure of $(Y\times Y, S\times S)$ with positive measure-theoretic $\lambda$-entropy and $\lambda (X\setminus \Delta_Y)> 0$
 by \cite{Glasner1997}, where $\Delta_Y= \{(y, y): y\in Y\}$ and $X\subset Y\times Y$ is the support of $\lambda$, that is, $X$ is the smallest
  closed subset of $Y\times Y$ with $\lambda (X)= 1$. It is easy to see that $(X, S\times S)$ forms a transitive system having a nonempty
   intersection with $\Delta_Y$, denoted by $(X, T)$. Then $X\supsetneq \Delta_Y$ by the minimality of $(Y, S)$. Additionally, if $(y_1, y_2)\in X$ then $\pi_{\text{eq}} (y_1)= \pi_{\text{eq}} (y_2)$, and hence $y_1$ and $y_2$ are proximal (as $\pi_{\text{eq}}$ is a proximal extension). In particular, if $(y_1, y_2)\in X$ is a minimal point of $(X, T)$ then $y_1= y_2$, and so $\Delta_Y$ is the unique minimal subsystem contained in $(X, T)$.

   Now we check that the system $(X, T)$ satisfies all required properties.
Clearly, $(X, T)$ is a nonminimal E-system. Moreover, it is an almost one-to-one extension of a
minimal equicontinuous system $(Y_{\text{eq}}, S_{\text{eq}})$ (because $\pi_{\text{eq}}$ is almost one-to-one and $\pi_{\text{eq}} (y_1)= \pi_{\text{eq}} (y_2)$ for all $(y_1, y_2)\in X$), and hence not thickly sensitive by Proposition \ref{201603271142}.
In fact, we can also obtain it directly from the definitions.

Now let us check $\text{Eq}_{\text{syn}} (X, T)= X$.
Let $x\in X$ and $y_0\in Y$, and hence
 $y_0\in \text{Eq}_{\text{syn}} (Y, S)$ by Theorem \ref{1407181746}. Take a compatible metric $\varrho_1$ over $Y$, and a compatible metric $\varrho ((y_1, y_2), (y_1', y_2'))=
\max \{\varrho_1 (y_1, y_1'), \varrho_1 (y_2, y_2')\}$ over $X$ for $y_1, y_1', y_2, y_2'\in Y$.
Then for each $\delta> 0$ there exists open $U_Y\subset Y$ containing $y_0$ and
 syndetic $\mathcal{N}\subset \mathbb{N}$ such that $\varrho_1 (S^n y, S^n y_0)\le \delta$ whenever $y\in U_Y$ and $n\in \mathcal{N}$.
  Thus $\varrho_1 (S^n y_1, S^n y_2)\le 2 \delta$ whenever $y_1, y_2\in U_Y$ and $n\in \mathcal{N}$, and finally
  $\varrho (T^n x_1, T^n x_2)\le 2 \delta$ whenever $x_1, x_2\in (U_Y\times U_Y)\cap X$ and $n\in \mathcal{N}$.
As $\Delta_Y$ is the unique
 minimal subsystem of $(X, T)$, there exist $m\in \mathbb{N}$ and open $U\subset X$ containing $x$ with $T^m U\subset (U_Y\times U_Y)\cap X$,
and so $\varrho (T^{m+ n} x, T^{m+ n} x')\le 2 \delta$ whenever $x'\in U$ and $n\in \mathcal{N}$. This implies
   $x\in \text{Eq}_{\text{syn}} (X, T)$, because $m+ \mathcal{N}\subset \mathbb{N}$ is syndetic. The construction is done.
\end{proof}

\begin{exam} \label{1407142300}
There exists a  nonminimal E-system $(X', T')$ which
is not thickly sensitive, such that
$\Tran (X', T')\subsetneq
\text{Eq}_{\text{syn}} (X', T')\subsetneq X'$.
\end{exam}
\begin{proof}[Construction]
Let $(X, T)$ and $(Y, S)$ be the systems as constructed in Example \ref{1407142346}, and we take $(X', T')$ to be the
system constructed by collapsing $\Delta_Y$ into a fixed point $p_0$ of $(X, T)$. Then the fixed point $p_0$ is the unique minimal point of $(X', T')$. Let $\pi: (X, T)\rightarrow (X', T')$ be the
 corresponding factor map.

    Now we check that the system $(X', T')$ satisfies all required properties.
It is easy to see that $(X', T')$ is an invertible nonminimal E-system.
 Then $X'\setminus \Tran (X', T')$ is a dense subset of $X'$ (see \cite{KS1997}), and hence
  $X'\setminus \Tran (X', T')\supsetneq \{p_0\}$.
  Moreover, the factor map $\pi$ is almost open.
 This implies that $(X', T')$ is not thickly sensitive, because $(X, T)$ is not thickly sensitive.
In fact, $\pi: X\setminus \Delta_Y\rightarrow X'\setminus \{p_0\}$ is a
   homeomorphism. Therefore we obtain that $\Tran (X', T')\subsetneq X'\setminus \{p_0\}\subset \text{Eq}_{\text{syn}} (X', T')$, as
   $\text{Eq}_{\text{syn}} (X, T)= X$.

Finally we are going to show that $p_0\notin \text{Eq}_{\text{syn}} (X', T')$ and hence $\Tran (X', T')\subsetneq X'\setminus \{p_0\}=
\text{Eq}_{\text{syn}} (X', T')\subsetneq X'$. Let $x_*\in X'\setminus \{p_0\}$ and set
 $0< \delta< \dist (\{x_*\}, \{p_0\})$. Let $U_0\subset X$ be any open set containing $p_0$ and $m\in \mathbb{N}$. By shrinking $U_0$ we may choose open $U_*$ containing $x_*$ such that $\text{dist} (U_*, U_0)> \delta$. We take opene $W\subset U_0$ such that $(T')^{- j} W\subset U_0$ for all $j= 0, 1, \dots, m$ (as $T' p_0= p_0$), and then
\begin{eqnarray*}
N_{(T')^{- 1}} (U_*, U_0)&\supset & \bigcup_{j= 0}^m N_{(T')^{- 1}} (U_*, (T')^{- j} W)\ \\
&\supset & \{n+ j: n\in N_{(T')^{- 1}} (U_*, W), j= 0, 1, \dots, m\}.
\end{eqnarray*}
Thus $N_{T'} (U_0, U_*)= N_{(T')^{- 1}} (U_*, U_0)$ is a thick set, because $N_{(T')^{- 1}} (U_*, W)\neq \varnothing$ by the transitivity of $(X', (T')^{- 1})$ (as $(X', T')$ is an invertible transitive system, $(X', (T')^{- 1})$
 is also transitive from the definition).
And hence $S_{T'} (U_0, \delta)$ is a thick set, as $T' p_0= p_0$. In particular,
  $p_0\notin \text{Eq}_{\text{syn}} (X', T')$.
The construction is done.
\end{proof}

Theorem 3.4 from \cite{KST2001} says that any irrational rotation of the two dimensional torus has an almost
one-to-one extension which is a noninvertible minimal map of that torus. By Proposition \ref{201603271142} any of these above mentioned noninvertible minimal maps on the torus is
syndetically equicontinuous, and by the Auslander-Yorke Dichotomy Theorem
any noninvertible minimal map is sensitive. But we still have the following open question.

\begin{quest}
If a syndetically equicontinuous system is uniformly rigid then is it equicontinuous?
\end{quest}

\section{Proof of Theorem \ref{1407120539}}

In this section we present a proof of our dichotomy Theorem \ref{1407120539}.

\begin{lem} \label{1407070139}
Let $\pi: (X, T)\rightarrow (Y, S)$ be a factor map and let $y_0\in \text{Eq} (Y, S)$ be a minimal point such that $\pi^{- 1} (y_0)= \{x_0\}$. Then $x_0\in \text{Eq}_{\text{syn}} (X, T)$. In particular, $(X, T)$ is not thickly sensitive.
\end{lem}
\begin{proof}
Let $\delta> 0$ and let $\varrho_Y$ be a compatible
 metric over $Y$. We take open $W\subset X$ containing $x_0$ such that
  $\diam (W)< \delta$. As $\{x_0\}= \pi^{- 1} (y_0)$, there exists open $V\subset Y$ containing $y_0$ such that $\pi^{- 1} (V)\subset W$.
   Let $\varepsilon> 0$ be small enough such that $\{y\in Y: \varrho_Y (y_0, y)< 2 \varepsilon\}\subset V$. Since $y_0\in \text{Eq} (Y, S)$,
    there exists $\varepsilon\ge \kappa> 0$ such that $\varrho_Y (S^n y, S^n y_0)< \varepsilon$ whenever $\varrho_Y (y, y_0)< \kappa$
    and $n\in \mathbb{N}$.
Take $V'= \{y\in Y: \varrho_Y (y_0, y)< \kappa\}, U= \pi^{- 1} (V')\ni x_0$ and set $\mathcal{S}= N_S (y_0, V')$.
Note that $\mathcal{S}$ is syndetic, because $y_0$ is a minimal point.
Let
$n\in \mathcal{S}$. Then $S^n y_0\in V'$, additionally, if $y\in V'$ then $\varrho_Y (S^n y_0, S^n y)< \varepsilon$ and so $\varrho_Y (y_0, S^n y)< 2 \varepsilon$,
 that gives $S^n V'\subset V$, and hence
$$T^n U= T^n \pi^{- 1} (V')\subset \pi^{- 1} (S^n V')\subset \pi^{- 1} (V)\subset W.$$

Summing up, for each $\delta> 0$ there exist open $U\subset X$ containing $x_0$ and a syndetic set $\mathcal{S}\subset \mathbb{N}$ such that $\varrho (T^n x_0, T^n x)< \delta$ for all $x\in U$ and any $n\in \mathcal{S}$ (recall that $\varrho$ is the metric over $X$).
That is, $x_0\in \text{Eq}_{\text{syn}} (X, T)$.
\end{proof}

Similarly, we can provide proof
of the following result, which is in fact Lemma 6.3 from \cite{HKKZ2016}.


\begin{prop} \label{201603271142}
If $(X, T)$ is an almost one-to-one extension of a minimal equicontinuous system, then
it is syndetically equicontinuous.
\end{prop}

In order to prove the another direction of Theorem \ref{1407120539}, we make some preparations.
Let $\pi: (X, T)\rightarrow (Y, S)$ be a factor map between dynamical systems. Recall that $\pi$ is proximal
if any pair of points $x_1, x_2\in X$ is proximal for $(X,T)$ whenever $\pi (x_1)= \pi (x_2)$.  The following result should be well known,
but we fail to find a reference and hence provide a proof of it here for completeness. This variant of a
very short proof was kindly communicated to us by Joseph Auslander and works for systems of any acting group.

\begin{lem} \label{added}
Let $\pi_i: (X_i, T_i)\rightarrow (Y_i, S_i), i= 1, 2$ be factor maps between dynamical systems. If both $\pi_1$ and $\pi_2$ are
proximal, then the product factor map $\pi_1\times \pi_2: (X_1\times X_2, T_1\times T_2)\rightarrow (Y_1\times Y_2, S_1\times S_2)$
is also proximal.
\end{lem}
\begin{proof}
Let a pair of points $(x_1, x_2), (x_1^*, x_2^*)\in X_1\times X_2$ with $\pi_i (x_i)= \pi_i (x_i^*)$ for $i= 1, 2$. We will show that
this pair is proximal. In fact, take a minimal point
$((z_1, z_2), (z_1^*, z_2^*))$ from the orbit closure of $((x_1, x_2), (x_1^*, x_2^*))$ in the system $(X_1\times X_2\times X_1\times X_2, T_1\times T_2\times T_1\times T_2)$. Then $(z_1, z_1^*)$ is a minimal point in the system $(X_1\times X_1, T_1\times T_1)$, and $\pi_1 (z_1)= \pi_1 (z_1^*)$ and hence $(z_1, z_1^*)$ is proximal, which implies $z_1= z_1^*$. Similarly $z_2= z_2^*$. In particular, $(x_1, x_2)$ and $(x_1^*, x_2^*)$ are proximal.
\end{proof}

%

Then we have the following

\begin{prop} \label{1407150743}
Let $\pi: (X, T)\rightarrow (Y, S)$  be a factor, not almost one-to-one map between minimal systems, where $(Y, S)$ is invertible.
Then $\inf\limits_{y\in Y} \diam (\pi^{- 1} y)> 0$. Moreover, if $\pi$ is also proximal,
then $(X, T)$ is thickly sensitive.
\end{prop}
\begin{proof}
As $(Y, S)$ is an invertible minimal system, it is not hard to show that $\pi^{- 1} (y)$ is
not a singleton for any $y\in Y$. Let us first prove that $d:=\inf\limits_{y\in Y} \diam (\pi^{- 1} y)> 0$.

Let $\psi: Y\rightarrow [0,\diam(X)]$ be given by $y\mapsto \diam (\pi^{-1} y)$, and hence for
 each $y\in Y$ one has $\psi (y)> 0$ as $\pi^{- 1} (y)$ is not a singleton.
Since the function $\psi$ is upper semi-continuous, $E_c(\psi)$ - the set of all points of continuity of
$\psi$, is a residual subset of $Y$ (see for example \cite[Lemma 1.28]{Furstenberg1981}). Suppose that $d= 0$.
So, there exists a sequence of points $y_i\in Y$  such that  $\lim\limits_{i\rightarrow \infty} \psi (y_i)= 0$.

Let $y_c\in E_c(\psi)$
and $\varepsilon> 0$. There exists
 open $V\subset Y$ containing $y_c$ such that $|\psi (y_c)- \psi (y)|\le \varepsilon$ whenever $y\in V$. Since $(Y, S)$ is minimal, there
 exists $m\in \mathbb{N}$ with $\bigcup\limits_{j= 0}^m S^{- j} V= Y$. By taking a subsequence (if necessary) we may
 assume that $\{y_i: i\in \mathbb{N}\}\subset S^{-k} V$ for some $k\in \{0, 1, \dots, m\}$. Since $\lim\limits_{i\rightarrow
 \infty} \psi (y_i)= 0$, in other words, the diameter of $\pi^{- 1} (y_i)$ tends to zero,  the diameter of
$\pi^{- 1} (S^k y_i)= T^k \pi^{- 1} (y_i)$ also tends to zero. Therefore $\lim\limits_{i\rightarrow \infty} \psi (S^k y_i)= 0$,
which implies that $\psi (y_c)\le \varepsilon$ by the construction of $V$ and $m$. Finally $\psi (y_c)= 0$, a contradiction.

Take $0< \delta< \dfrac{d}{6}$. Now assume that $\pi$ is  proximal. We shall prove that $(X, T)$ is thickly sensitive with a sensitivity constant $\delta$. Let $x_*\in X$
and $m\in \mathbb{N}$ and take open $U\subset X$ containing $x_*$. Let  $V\subset U$ be an open set containing $x_*$ with
$\max\limits_{0\le i\le m} \diam (T^i V)< \delta$. Since a factor map between minimal systems is almost open
\cite[Theorem 1.15]{Auslander1988}, therefore for each $i= 0, 1, \dots, m$ we can choose $y_i\in \inter (\pi (T^i V))$
(the interior of $\pi (T^i V)$),  $u_i\in \pi^{- 1} (y_i)$ with $\dist (u_i, T^i V)> \dfrac{d}{2}- \delta$
(because $\diam (T^i V)< \delta$), and set
$$W_i= \{x\in X: \varrho (x, u_i)< \delta\}\cap \pi^{- 1} (\inter (\pi (T^i V)))\ni u_i.$$
Obviously, $\dist (W_i, T^i V)> \dfrac{d}{2}- 2 \delta> \delta$ for each $i= 0, 1, \dots, m$.

Note that since $(X, T)$ is minimal, the set of all minimal points
of the system $(X^{m+ 1}, T^{(m+ 1)})$, the product system of $m+1$ copies of $(X, T)$,  is dense in $X^{m+ 1}$.  Hence we can take a minimal point
$(v_0, v_1, \dots, v_m)\in W_0\times W_1\times \dots\times W_m$ of the system $(X^{m+ 1}, T^{(m+ 1)})$, and
let $x_i\in T^i V$ with $\pi (x_i)= \pi (v_i)$ (because $\pi (v_i)\in \pi (T^i V)$) for
 each $i= 0, 1, \dots, m$. Since the factor map
 $\pi: (X, T)\rightarrow (Y, S)$ is proximal,  by Lemma \ref{added},
$$\pi': (X^{m+ 1}, T^{(m+ 1)})\rightarrow (Y^{m+ 1}, S^{(m+ 1)}), (x_i': 0\le i\le m)\mapsto (\pi (x_i'): 0\le i\le m)$$
is also a proximal factor map. In particular, $((x_0, x_1, \dots, x_m), (v_0, v_1, \dots, v_m))$ is proximal (under the
action $T^{(m+ 1)}$), and thus
$$\mathcal{S}= N_{T^{(m+ 1)}} ((x_0, x_1, \dots, x_m), W_0\times W_1\times \dots\times W_m)$$
is a central set and  contains a $\Delta$-set by \cite{Furstenberg1981}. Finally
$\mathcal{S}\cap \mathcal{N}\neq \varnothing $ where $\mathcal{N}= N_T (V, V)\subset N_T (T V, T V)\subset \dots \subset N_T (T^m V, T^m V)$
is a $\Delta^*$-set \cite{Furstenberg1981}. Now for any $n\in \mathcal{S}\cap \mathcal{N}$ and each $i= 0, 1, \dots, m$, on one hand
 $T^n x_i\in W_i$ as $n\in \mathcal{S}$, and hence $T^{n+ i} V\cap W_i\ni T^n x_i$ as $x_i\in T^i V$; on the other hand
 $T^{n+ i} V\cap T^i V\neq \varnothing $ as $n\in \mathcal{N}$, therefore $\diam (T^{n+ i} V)\ge \dist (W_i, T^i V)> \delta$.
  Thus
$$S_T (U, \delta)\supset S_T (V, \delta)\supset \{n+ i: n\in \mathcal{S}\cap \mathcal{N}, i= 0, 1, \dots, m\},$$
which implies that $(X, T)$ is thickly sensitive by the arbitrariness of $U$ and $m$.
\end{proof}


The following lemma  is just a reformulation of \cite[Theorem 1.1]{Veech1968}.

\begin{lem} \label{1407160355}
Let $(X, T)$ be an invertible minimal system and $x, x'\in X$. Then $(x, x')\in Q_+ (X, T)$ if and only if for every open
 $U, V\subset X$ containing $x$ and  $x'$, respectively, there exist $n_1, m_1\in \mathbb{Z}$ such that
  $T^{n_1} x, T^{n_1+ m_1} x\in U$ and $T^{m_1} x\in V$.
\end{lem}

Now let us show that it is also true for any (not only invertible) continuous minimal map. Recall that $\mathcal{S}\subset \mathbb{N}$ is an \emph{IP set} if there exists $\{p_k: k\in \mathbb{N}\}\subset \mathbb{N}$ with $\{p_{i_1}+ \dots+ p_{i_k}: k\in \mathbb{N}\
 \text{and}\ i_1< \dots< i_k\}\subset \mathcal{S}$.

\begin{lem} \label{1407160403}
Let $(X, T)$ be a minimal system and $x, y\in X$. Then $(x, y)\in Q_+ (X, T)$ if and only if for every open
 $U,V \subset X$ containing $x$ and $y$, respectively, there exist $n, m\in \mathbb{N}$ such that
  $T^n x, T^{n+ m} x\in U$ and $T^m x\in V$.
\end{lem}
\begin{proof}
Firstly assume that, for every open
 $U, V\subset X$ containing $x$ and $y$, respectively, there exist $n, m\in \mathbb{N}$ such that
  $T^n x, T^{n+ m} x\in U$ and $T^m x\in V$. From the definition it is readily to obtain $(x, y)\in Q_+ (X, T)$, for instance by taking
  $x'=x, y'= T^m(x)$.

  Now assume $(x, y)\in Q_+ (X, T)$ and take open
 $U, V\subset X$ containing $x$ and $y$, respectively.
Recall that $(\widehat{X}, \widehat{T})$ is the natural extension of $(X, T)$ and $\widehat{\pi}: (\widehat{X}, \widehat{T})\rightarrow (X, T)$ is the associated factor map. Hence
$Q_+ (X, T)=
  (\widehat{\pi}\times \widehat{\pi}) Q_+ (\widehat{X}, \widehat{T})$
 by \cite[Lemma A.3]{HuangYe2000}.
 In particular, there exist $(x_*, y_*)\in Q_+ (\widehat{X}, \widehat{T})$ and open $U_*, V_* \subset \widehat{X}$
 containing $x_*$ and $y_*$, respectively, such that $\widehat{\pi} (x_*)= x$, $\widehat{\pi} (y_*)= y$ and
 $\widehat{\pi} (U_*)\subset U$, $\widehat{\pi} (V_*)\subset V$.
 Since $(X, T)$ is minimal, $(\widehat{X}, \widehat{T})$
 is an invertible minimal system, and then by applying Lemma \ref{1407160355}
there exist $n_1, m_1\in \mathbb{Z}$ such that
  $\widehat{T}^{n_1} x_*, \widehat{T}^{n_1+ m_1} x_*\in U_*$ and $\widehat{T}^{m_1} x_*\in V_*$. Moreover,
  we choose open $W\subset \widehat{X}$ containing $x_*$ such that $\widehat{T}^{n_1} W\subset U_*, \widehat{T}^{n_1+ m_1} W\subset U_*$
   and $\widehat{T}^{m_1} W\subset V_*$. Since $(\widehat{X}, \widehat{T})$ is minimal, $x_*$ is recurrent in the sense that
     $\widehat{T}^{l_k}
   x_*$ tends to $x_*$ for a sequence of positive integers $l_1< l_2< \dots$, and so $N_{\widehat{T}}
   (x_*, W)$ is an IP set by \cite[Theorem 2.17]{Furstenberg1981}. Hence there exist $p_1, q_1\in \mathbb{N}$ such that
   $$n= n_1+ p_1> 0, m= m_1+ q_1> 0\ \text{and}\ \{p_1, q_1, p_1+ q_1\}\subset N_{\widehat{T}}
   (x_*, W).$$
Thus $\widehat{T}^n x_*, \widehat{T}^{n+ m} x_*\in U_*$ and $\widehat{T}^m x_*\in V_*$. Therefore $T^n x, T^{n+ m} x\in U$ and $T^m x\in V$
by the above construction. This finishes the proof.
\end{proof}

With the help of Lemma \ref{1407160403}, using an idea of the proof of \cite[Lemma 2.1.2]{Veech1965} we obtain the following result,
which is of independent interest.

\begin{prop} \label{1407160718}
Let $(X, T)$ be a minimal system and $x, y\in X$. Then $(x, y)\in Q_+ (X, T)$ if and only if $N_T (x, U)$ contains a $\Delta$-set for any
open $U\subset X$ containing $y$.
\end{prop}
\begin{proof}
\textit{Sufficiency}. Let $U\subset X$ be an open set containing $y$. Since $N_T (x, U)$ contains a $\Delta$-set by the assumption, there exists
 $\{s_1< s_2< s_3\}\subset
\mathbb{N}$ with $T^{s_3- s_2} x, T^{s_2- s_1} x, T^{s_3- s_1} x\in U$. Let $x'= x, y'= T^{s_2- s_1} x$ and $m= s_3- s_2\in \mathbb{N}$. Then
$T^m x',
T^m y'\in U$ and $(x, y)\in Q_+ (X, T)$ by the arbitrariness of $U$.

\textit{Necessity}. Assume  $(x, y)\in Q_+ (X, T)$ and take open $U\subset X$ containing $y$. Choose positive real numbers $\eta$ and
$\eta_k, k\in \mathbb{N}$ such that $\eta= \sum\limits_{k\in \mathbb{N}} \eta_k$ and $B_\eta (y)\subset U$, where $B_\eta (y)$ denotes
the open ball
of radius $\eta$ centered at $y$. By applying Lemma \ref{1407160403} to $B_{\eta_1} (x)$ and $B_{\eta_1} (y)$, there exist
 $n_1, m_1\in
\mathbb{N}$ such that
\begin{equation*} \label{1407161728}
T^{n_1} x, T^{n_1+ m_1} x\in B_{\eta_1} (x)\ \text{and}\ T^{m_1} x\in B_{\eta_1} (y).
\end{equation*}
Fix a $\delta> 0$. Applying Lemma \ref{1407160403}
   to $B_{\delta} (x)$ and $B_{\eta_1} (y)$, we have  $n_2, m_2\in
\mathbb{N}$ such that
  $T^{n_2} x, T^{n_2+ m_2} x\in B_{\delta} (x)$ and $T^{m_2} x\in B_{\eta_2} (y)$. Since $\delta$ can be selected small enough, we can require additionally
\begin{equation*} \label{1407161729}
\max_{0\le r\le n_1+ m_1} \varrho (T^{r+ n_2} x, T^r x)< \eta_2\ \text{and}\ \max_{0\le r\le n_1+ m_1} \varrho (T^{r+ n_2+ m_2} x, T^r x)< \eta_2.
\end{equation*}
We continue the process by induction. Put $l_k= \sum\limits_{i= 1}^k (n_i+ m_i)$ for each $k\in \mathbb{N}$. Then there exist
$n_{k+ 1}, m_{k+ 1}\in
\mathbb{N}$ such that
$T^{m_{k+ 1}} x\in B_{\eta_{k+ 1}} (y)$,
\begin{equation} \label{1407161734}
\max_{0\le r\le l_k} \varrho (T^{r+ n_{k+ 1}} x, T^r x)< \eta_{k+ 1}\ \text{and}\ \max_{0\le r\le l_k} \varrho (T^{r+ n_{k+ 1}+ m_{k+ 1}} x, T^r x)< \eta_{k+ 1}.
\end{equation}
Set $p_k= m_k+ n_{k+ 1}$ and $s_k= p_1+ \dots+ p_k$ for every $k\in \mathbb{N}$. Then, for all $i\le j$,
\begin{eqnarray*}
\varrho (T^{p_i+ \dots+ p_j} x, y)&= & \varrho (T^{m_i+ \sum\limits_{k= i+ 1}^j (n_k+ m_k) + n_{j+ 1}} x, y)\\
&\le & \varrho (T^{m_i+ \sum\limits_{k= i+ 1}^j (n_k+ m_k) + n_{j+ 1}} x, T^{m_i+ \sum\limits_{k= i+ 1}^j (n_k+ m_k)} x)+ \\
& & \dots+ \varrho (T^{m_i+ (n_{i+ 1}+ m_{i+ 1})} x, T^{m_i} x)+ \varrho (T^{m_i} x, y)\\
&< & \eta_{j+ 1}+ \dots+ \eta_i< \eta\ (\text{using \eqref{1407161734}}).
\end{eqnarray*}
So, $N_T (x, U)$ contains the $\Delta$-set $\{s_j- s_i: i< j\}$ from the construction.
\end{proof}

Recall that $\pi_{\text{eq}}: (X, T)\rightarrow (X_{\text{eq}}, T_{\text{eq}})$ is the corresponding factor map of $(X, T)$ over its maximal equicontinuous factor.

\begin{prop} \label{1407150651}
Let $(X, T)$ be a minimal system. Assume that $\pi_{\text{eq}}: (X, T)\rightarrow (X_{\text{eq}}, T_{\text{eq}})$
is not proximal. Then $(X, T)$ is thickly sensitive.
\end{prop}
\begin{proof} Since $\pi_{\text{eq}}: (X, T)\rightarrow (X_{\text{eq}}, T_{\text{eq}})$
is not proximal, there exists a pair of points $x_1, x_2\in X$, which is not proximal, such that $\pi_{\text{eq}} (x_1)= \pi_{\text{eq}} (x_2)$ (and hence $(x_1, x_2)\in Q_+ (X, T)$,
as $(X, T)$ is minimal). Then $d:= \inf\limits_{n\in \mathbb{N}} \varrho (T^n x_1, T^n x_2)> 0$.

Take $0< \delta< \dfrac{d}{3}$.
We are going to prove that  $(X, T)$ is thickly sensitive with a sensitivity constant $\delta> 0$. Since $(X, T)$ is minimal, it suffices to show
 that $S_T (U, \delta)$ is thick  for any open $U\subset X$ containing $x_1$.

 For any $m\in \mathbb{N}$ take open sets $V\subset U$ and $W$ containing $x_1$ and $x_2$, respectively, such that
$\max\limits_{0\le i\le m} \max \{\diam (T^i V), \diam (T^i W)\}< \delta$. By the above construction
 $\min\limits_{0\le i\le m} \dist (T^i V, T^i W)> \delta$. Since $(x_1, x_2)\in Q_+ (X, T)$, $N_T (x_1, W)$ contains a
$\Delta$-set by Proposition \ref{1407160718}, and hence has a nonempty intersection with $\mathcal{N}$, where
 $\mathcal{N}= N_T (V, V)\subset N_T (T V, T V)\subset \dots \subset N_T (T^m V, T^m V)$ is a $\Delta^*$-set by \cite[Page 177]{Furstenberg1981}.
 Therefore for every $n\in N_T (x_1, W)\cap \mathcal{N}$ and  $i= 0, 1, \dots, m$ we have: $T^{n+ i} V\cap T^i W\ni T^{n+ i} x_1$, because $T^n x_1\in W$;
 and $T^{n+ i} V\cap T^i V\neq \varnothing $, because  $n\in \mathcal{N}$. That gives  $\diam
  (T^{n+ i} V)\ge \dist (T^i W, T^i V)> \delta$. Thus
$$S_T (U, \delta)\supset S_T (V, \delta)\supset \{n+ i: n\in N_T (x_1, W)\cap \mathcal{N}, i= 0, 1, \dots, m\},$$
which implies that $(X, T)$ is thickly sensitive.
\end{proof}

Now we are ready to prove Theorem \ref{1407120539}.

\begin{proof}[Proof of Theorem \ref{1407120539}]
If $\pi_{\text{eq}}$ is almost one-to-one, then $(X, T)$ is not thickly sensitive by Proposition \ref{201603271142}.
Now assume that $(X, T)$ is not thickly sensitive, then $\pi_{\text{eq}}$ is proximal by Proposition
 \ref{1407150651}, and then $\pi_{\text{eq}}$ is almost one-to-one by Proposition \ref{1407150743} (as $(X_{\text{eq}}, T_{\text{eq}})$
 is an invertible minimal system). This finishes the proof.
\end{proof}

As a corollary of Theorem \ref{1407120539}, we have the following

\begin{prop} \label{201603271948}
Let $\pi: (X, T)\rightarrow (Y, S)$ be an almost one-to-one factor map between minimal systems. Then $(X, T)$ is syndetically equicontinuous if and only if so is $(Y, S)$.
\end{prop}
\begin{proof}
By Theorem \ref{1407181746}, it suffices to prove that $(X, T)$ is not thickly sensitive if and only if so is $(Y, S)$.
As a factor map between minimal systems, $\pi$ is almost open by \cite[Theorem 1.15]{Auslander1988}, and so if $(X, T)$ is not thickly sensitive then so is $(Y, S)$, as the thick sensitivity can be lifted from a factor to an extension by an almost open factor map by the method used in the
   proof of \cite[Lemma 1.6]{GlasnerWeiss1993}.

   Now assume that $(Y, S)$ is not thickly sensitive, and then the factor map $\pi^*_{\text{eq}}: (Y, S)\rightarrow (Y_{\text{eq}}, S_{\text{eq}})$, the factor map of $(Y, S)$ over its maximal equicontinuous factor
  $(Y_{\text{eq}}, S_{\text{eq}})$, is almost one-to-one
by Theorem \ref{1407120539}. Set $\pi^*$ to be the composition factor map $\pi^*_{\text{eq}}\circ \pi: (X, T)\rightarrow (Y_{\text{eq}}, S_{\text{eq}})$. Denote by $Y_1$ ($Y_2$, $Y_0$, respectively) the set of all points $y_1\in Y_{\text{eq}}$ ($y_2\in Y_{\text{eq}}$, $y_0\in Y$, respectively) whose fibers $(\pi^*)^{- 1} (y_1)$ ($(\pi^*_{\text{eq}})^{- 1} (y_2)$, $\pi^{- 1} (y_0)$, respectively) are singletons.
 Then $Y_2$ is a dense $G_\delta$ subset of $Y_{\text{eq}}$, and $Y_0$ is a dense $G_\delta$ subset of $Y$. This implies that $Y_0\cap (\pi^*_{\text{eq}})^{- 1} (Y_2)$ is a dense $G_\delta$ subset of $Y$ by Lemma \ref{201603271750}, and then $\pi^*_{\text{eq}} (Y_0\cap (\pi^*_{\text{eq}})^{- 1} (Y_2))$ is a dense subset of $Y_{\text{eq}}$. Note that $\pi^*_{\text{eq}} (Y_0\cap (\pi^*_{\text{eq}})^{- 1} (Y_2))\subset Y_1$. In fact, for any $y_*\in \pi^*_{\text{eq}} (Y_0\cap (\pi^*_{\text{eq}})^{- 1} (Y_2))\subset Y_2$, we take $y_0\in Y_0$ with $\pi^*_{\text{eq}} (y_0)= y_*$, then $(\pi^*_{\text{eq}})^{- 1} (y_*)= \{y_0\}$ as $y_*\in Y_2$, and hence
 $$(\pi^*)^{- 1} (y_*)= \pi^{- 1}\circ (\pi^*_{\text{eq}})^{- 1} (y_*)= \pi^{- 1} (y_0)$$
 is a singleton as $y_0\in Y_0$. Thus we have the denseness of $Y_1$ in $Y_{\text{eq}}$, and then $(X, T)$ is not thickly sensitive by Proposition \ref{201603271142}. This finishes the proof.
\end{proof}

\bibliographystyle{amsplain}


\begin{thebibliography}{10}

\bibitem{Akin1997}
Ethan Akin, \emph{Recurrence in topological dynamics}, The University Series in
  Mathematics, Plenum Press, New York, 1997, Furstenberg families and Ellis
  actions. \MR{1467479 (2000c:37014)}

\bibitem{AAB1993}
Ethan Akin, Joseph Auslander, and Kenneth Berg, \emph{When is a transitive map
  chaotic?}, Convergence in ergodic theory and probability ({C}olumbus, {OH},
  1993), Ohio State Univ. Math. Res. Inst. Publ., vol.~5, de Gruyter, Berlin,
  1996, pp.~25--40. \MR{1412595 (97i:58106)}

\bibitem{AkinGlasner2001}
Ethan Akin and Eli Glasner, \emph{Residual properties and almost
  equicontinuity}, J. Anal. Math. \textbf{84} (2001), 243--286. \MR{1849204
  (2002f:37020)}

\bibitem{AK}
Ethan Akin and Sergi{\u\i} Kolyada, \emph{Li-{Y}orke sensitivity}, Nonlinearity
  \textbf{16} (2003), no.~4, 1421--1433. \MR{1986303 (2004c:37016)}

\bibitem{Auslander1988}
Joseph Auslander, \emph{Minimal flows and their extensions}, North-Holland
  Mathematics Studies, vol. 153, North-Holland Publishing Co., Amsterdam, 1988,
  Notas de Matem{\'a}tica [Mathematical Notes], 122. \MR{956049 (89m:54050)}

\bibitem{AuslanderYorke}
Joseph Auslander and James~A. Yorke, \emph{Interval maps, factors of maps, and
  chaos}, T\^ohoku Math. J. (2) \textbf{32} (1980), no.~2, 177--188. \MR{580273
  (82b:58049)}

\bibitem{BHM2000}
Fran{\c{c}}ois Blanchard, Bernard Host, and Alejandro Maass, \emph{Topological
  complexity}, Ergodic Theory Dynam. Systems \textbf{20} (2000), no.~3,
  641--662. \MR{1764920 (2002b:37019)}

\bibitem{Bochner}
Salomon Bochner, \emph{Curvature and {B}etti numbers in real and complex vector
  bundles}, Univ. e Politec. Torino. Rend. Sem. Mat. \textbf{15} (1955--56),
  225--253. \MR{0084160 (18,819c)}

\bibitem{Bo}
\bysame, \emph{A new approach to almost periodicity}, Proc. Nat. Acad. Sci.
  U.S.A. \textbf{48} (1962), 2039--2043. \MR{0145283 (26 \#2816)}

\bibitem{Downarowicz2005}
Tomasz Downarowicz, \emph{Survey of odometers and {T}oeplitz flows}, Algebraic
  and topological dynamics, Contemp. Math., vol. 385, Amer. Math. Soc.,
  Providence, RI, 2005, pp.~7--37. \MR{2180227 (2006f:37009)}

\bibitem{EG1960}
Robert Ellis and Walter~H. Gottschalk, \emph{Homomorphisms of transformation
  groups}, Trans. Amer. Math. Soc. \textbf{94} (1960), 258--271. \MR{0123635
  (23 \#A960)}

\bibitem{Furstenberg1981}
Harry Furstenberg, \emph{Recurrence in ergodic theory and combinatorial number
  theory}, Princeton University Press, Princeton, N.J., 1981, M. B. Porter
  Lectures. \MR{603625 (82j:28010)}

\bibitem{Glasner1997}
Eli Glasner, \emph{A simple characterization of the set of {$\mu$}-entropy
  pairs and applications}, Israel J. Math. \textbf{102} (1997), 13--27.
  \MR{1489099 (98k:54076)}

\bibitem{GlasnerWeiss1993}
Eli Glasner and Benjamin Weiss, \emph{Sensitive dependence on initial
  conditions}, Nonlinearity \textbf{6} (1993), no.~6, 1067--1075. \MR{1251259
  (94j:58109)}

\bibitem{Gu}
John Guckenheimer, \emph{Sensitive dependence to initial conditions for
  one-dimensional maps}, Comm. Math. Phys. \textbf{70} (1979), no.~2, 133--160.
  \MR{553966 (82c:58037)}

  \bibitem{HKKZ2016}
Wen Huang, Danylo Khilko, Sergi{\u\i} Kolyada, and Guohua Zhang,
  \emph{Dynamical compactness and sensitivity}, J. Differential Equations
  \textbf{260} (2016), no.~9, 6800--6827. \MR{3461085}

\bibitem{HuangYe2000}
Wen Huang and Xiangdong Ye, \emph{Devaney's chaos or 2-scattering implies
  {L}i-{Y}orke's chaos}, Topology Appl. \textbf{117} (2002), no.~3, 259--272.
  \MR{1874089 (2003b:37017)}

\bibitem{KS1997}
Sergi{\u\i} Kolyada and L'ubom{\'{\i}}r Snoha, \emph{Some aspects of
  topological transitivity---a survey}, Iteration theory ({ECIT} 94) ({O}pava),
  Grazer Math. Ber., vol. 334, Karl-Franzens-Univ. Graz, Graz, 1997, pp.~3--35.
  \MR{1644768}

\bibitem{KST2001}
Sergi{\u\i} Kolyada, L'ubom{\'{\i}}r Snoha, and Serge{\u\i} Trofimchuk,
  \emph{Noninvertible minimal maps}, Fund. Math. \textbf{168} (2001), no.~2,
  141--163. \MR{1852739 (2002j:37017)}

\bibitem{LiYe}
Jian Li and Xiangdong Ye, \emph{Recent development of chaos theory in
  topological dynamics}, Acta Math. Sin. (Engl. Ser.) \textbf{32} (2016),
  no.~1, 83--114. \MR{3431162}

\bibitem{liuheng}
Heng Liu, Li~Liao, and Lidong Wang, \emph{Thickly {S}yndetical {S}ensitivity of
  {T}opological {D}ynamical {S}ystem}, Discrete Dyn. Nat. Soc. (2014), Art. ID
  583431, 4. \MR{3200824}

\bibitem{Subrahmonian2007}
T.~K.~Subrahmonian Moothathu, \emph{Stronger forms of sensitivity for dynamical
  systems}, Nonlinearity \textbf{20} (2007), no.~9, 2115--2126. \MR{2351026
  (2008j:37014)}

\bibitem{Ru}
David Ruelle, \emph{Dynamical systems with turbulent behavior}, Mathematical
  problems in theoretical physics ({P}roc. {I}nternat. {C}onf., {U}niv. {R}ome,
  {R}ome, 1977), Lecture Notes in Phys., vol.~80, Springer, Berlin-New York,
  1978, pp.~341--360. \MR{518445 (81e:58032)}

\bibitem{Veech1965}
William~A. Veech, \emph{Almost automorphic functions on groups}, Amer. J. Math.
  \textbf{87} (1965), 719--751. \MR{0187014 (32 \#4469)}

\bibitem{Veech1968}
\bysame, \emph{The equicontinuous structure relation for minimal {A}belian
  transformation groups}, Amer. J. Math. \textbf{90} (1968), 723--732.
  \MR{0232377 (38 \#702)}

\bibitem{Veech1977}
\bysame, \emph{Topological dynamics}, Bull. Amer. Math. Soc. \textbf{83}
  (1977), no.~5, 775--830. \MR{0467705 (57 \#7558)}

\bibitem{Walter}
Peter Walters, \emph{An introduction to ergodic theory}, Graduate Texts in
  Mathematics, vol.~79, Springer-Verlag, New York-Berlin, 1982. \MR{648108
  (84e:28017)}

\bibitem{YeYu}
Xiangdong Ye and Tao Yu, \emph{Sensitivity, proximal extension and higher order
  almost automorphy}, preprint, 2015.


\end{thebibliography}

\end{document}